\documentclass[11pt]{amsart}
\usepackage[margin=1.35in]{geometry}
\usepackage{amscd,amsmath,amsxtra,amsthm,amssymb,stmaryrd,xr,mathrsfs,mathtools,enumerate,commath, comment}
\usepackage{stmaryrd}
\usepackage{xcolor}
\usepackage{commath}
\usepackage{comment}
\usepackage{tikz-cd}
\usepackage{longtable} 
\usepackage{pdflscape} 
\usepackage{booktabs}
\usepackage{hyperref}
\definecolor{vegasgold}{rgb}{0.77, 0.7, 0.35}
\definecolor{darkgoldenrod}{rgb}{0.72, 0.53, 0.04}
\definecolor{gold(metallic)}{rgb}{0.83, 0.69, 0.22}
\hypersetup{
 colorlinks=true,
 linkcolor=darkgoldenrod,
 filecolor=brown,      
 urlcolor=gold(metallic),
 citecolor=darkgoldenrod,
 pdftitle={Structure },
 }

\usepackage[all,cmtip]{xy}

\DeclareFontFamily{U}{wncy}{}
\DeclareFontShape{U}{wncy}{m}{n}{<->wncyr10}{}
\DeclareSymbolFont{mcy}{U}{wncy}{m}{n}
\DeclareMathSymbol{\Sh}{\mathord}{mcy}{"58}
\usepackage[T2A,T1]{fontenc}
\usepackage[OT2,T1]{fontenc}

\newtheorem{Theorem}{Theorem}[section]
\newtheorem{Lemma}[Theorem]{Lemma}

\newtheorem{Conjecture}[Theorem]{Conjecture}

\newtheorem{Corollary}[Theorem]{Corollary}
\newtheorem{Definition}[Theorem]{Definition}
\newtheorem{Assumption}[Theorem]{Assumption}

\numberwithin{equation}{section}

\theoremstyle{remark}
\newtheorem{Remark}[Theorem]{Remark}

\newcommand{\Zx}{\mathbb{Z}_p\llbracket x\rrbracket}

\newcommand{\cF}{\mathcal{F}}

\newcommand{\Seln}{\op{Sel}_{p^\infty}(E/F^{(n)}_{\op{cyc}})}

\newcommand{\Z}{\mathbb{Z}}

\newcommand{\Q}{\mathbb{Q}}
\newcommand{\F}{\mathbb{F}}

\newcommand{\Fin}{F_\infty}
\newcommand{\Fcyc}{F_{\op{cyc}}}

\newcommand{\op}[1]{\operatorname{#1}}

\newcommand\mtx[4] { \left( {\begin{array}{cc}
 #1 & #2 \\
 #3 & #4 \\
 \end{array} } \right)}

\title[Mordell--Weil ranks in noncommutative towers]{Asymptotic growth of Mordell--Weil ranks of elliptic curves in noncommutative towers}

\author[Anwesh Ray]{Anwesh Ray}
\address[Ray]{Department of Mathematics\\
University of British Columbia\\
Vancouver BC, Canada V6T 1Z2}
\email{anweshray@math.ubc.ca}

\begin{document}

\maketitle

\begin{abstract}

Let $E$ be an elliptic curve defined over a number field $F$ with good ordinary reduction at all primes above $p$, and let $F_\infty$ be a finitely ramified uniform pro-$p$ extension of $F$ containing the cyclotomic $\Z_p$-extension $F_{\op{cyc}}$. Set $F^{(n)}$ be the $n$-th layer of the tower, and $F^{(n)}_{\op{cyc}}$ the cyclotomic $\Z_p$-extension of $F^{(n)}$. We study the growth of the rank of $E(F^{(n)})$ by analyzing the growth of the $\lambda$-invariant of the Selmer group over $F^{(n)}_{ \op{cyc}}$ as $n\rightarrow \infty$. This method has its origins in work of A.~Cuoco, who studied $\Z_p^2$-extensions. Refined estimates for growth are proved that are close to conjectured estimates. The results are illustrated in special cases.
\end{abstract}

\section{Introduction}
\par The \emph{Mordell--Weil Theorem} states that given an elliptic curve $E$ defined over a number field $F$, its $F$-rational points form a finitely generated abelian group, i.e.,
\[
E(F) \simeq \Z^r \oplus E(F)_{\op{tors}}
\]
where $r$ is a non-negative integer called the Mordell--Weil \emph{rank}. In \cite{mazur1972rational}, B.~Mazur initiated the study of Selmer groups of elliptic curves in $\Z_p$-extensions. A major application of Iwasawa theory is the study of the growth of Mordell--Weil ranks of abelian varieties in towers of number fields. Given an abelian variety defined over a number field $F$ with good ordinary reduction at the primes above $p$, Mazur showed that the rank of $A$ is bounded in the cyclotomic $\Z_p$-extension of $F$. K.~Kato and D.~Rohlrich proved the analogous statement for all elliptic curves defined over abelian number fields, see \cite{kato2004p, rohrlich1984onl}. Alongside further developments in Iwasawa theory over larger $p$-adic Lie extensions, there has been significant interest in analyzing the asymptotic growth of Mordell--Weil ranks in towers. For ordinary primes, Mazur in \cite{mazur1983modular} formulated a precise conjecture on the growth of ranks in any $\Z_p$-extension of an imaginary quadratic field $F$, called the \emph{Growth number conjecture}. This question has been studied in anticyclotomic extensions by C.~Cornut \cite{cornut2002mazur} and V.~Vatsal \cite{vatsal2003special}. A prototypical example of interest is the $\Z_p^2$-extension $F_\infty$ of an imaginary quadratic field $F$. For every integer $n\geq 1$, set $F^{(n)}$ to denote the $n$-th layer. In this setting, it is the extension contained in $F_\infty$ such that $\op{Gal}(F^{(n)}/F)=\left(\Z/p^n\Z\right)^2$. For elliptic curves $E_{/F}$, asymptotic formulas for the growth of the rank of $E(F^{(n)})$ as $n\rightarrow \infty$ have been proven by A.~Lei and F.~Sprung in \cite{lei2020ranks}. More recently, such growth questions are studied in admissible uniform pro-$p$ extensions of number fields by D.~Delbourgo and A.~Lei in \cite{delbourgo2017estimating}, and by P.~C.~Hung and M.~F.~Lim in \cite{HL}.

\par In this note, we employ a new strategy to study the growth of ranks in non-commutative towers. Let $F^{(n)}_{\op{cyc}}$ be the cyclotomic $\Z_p$-extension of $F^{(n)}$. We study the growth of the rank of $E(F^{(n)})$ by analyzing the growth of the $\lambda$-invariant of the Selmer group over $F^{(n)}_{ \op{cyc}}$ as $n\rightarrow \infty$ using a generalizations of Kida's formula due to Y.~Hachimori $\&$ K.~Matsuno \cite{HM}, and M.~F.~Lim \cite{lim2015remark}. The method has several advantages. First, it is a straightforward application of Kida's formula which a prioi allows one to circumvent technicalities of noncommutative Iwasawa theory. In other words, the proof is short and can be understood without familiarity with methods in noncommutative Iwasawa theory, though it does build on previous work of M.~F.~Lim \cite{lim2021some} which uses nontrivial results in the subject. Given an elliptic curve $E$ over a number field $F$, we do however require that the Selmer group of $E$ over the cyclotomic $\Z_p$-extension of $F$ be cotorsion over the Iwasawa algebra and impose the $\mathfrak{M}_H(G)$ conjecture (see \cite{CS12,lim2015remark}).
\par Secondly, (and perhaps more importantly) the method strengthens known results and these estimates are closer to conjectured asymptotics. The error term in the asymptotic formulas of Hung-Lim are removed in the process, and the main term is sharper (see Remark \ref{better estimate} for further details). The significance of this is illustrated for certain examples, namely, \emph{$\Z_p^d$-extensions}, \emph{false Tate-curve extensions} and \emph{trivializing extensions} generated by the $p$-primary torsion of a non-CM elliptic curve. It should be pointed out here that for $\Z_p^2$-extensions, a similar question was studied by A.~Cuoco in \cite{cuoco1980growth}, who studied the growth of Iwasawa invariants associated with class group towers in families of $\Z_p$-extensions contained in the composite of two $\Z_p$-extensions. The results can also be applied to prove statistical results.
\begin{enumerate}
    \item In Corollary \ref{100 percent}, it is shown that if $E$ is an elliptic curve defined over an imaginary quadratic field $F$ such that $E$ does not have complex multiplication and $\op{rank}E(F)=0$, then the rank remains 0 is $100\%$ of $\Z_p^2$-extensions of $F$.
    \item In Corollary $\ref{2ndapp}$, we consider the curve $E=\href{https://www.lmfdb.org/EllipticCurve/Q/11a2/}{11a2}$. It is shown that there is a positive density set of primes $\ell$ such that \[\op{rank}E\left(\Q(\mu_{7^{n+1}}, \ell^{\frac{1}{7^n}})\right)\leq \lambda_7(E/\Q(\mu_{7^\infty}))7^n,\] for all integers $n>0$.
\end{enumerate}
\par The method employed in this paper shows that in any context in which a satisfactory generalization of Kida's formula is proved, it should be possible to analyze the growth of $\lambda$-invariants in noncommutative towers. We point out that analogs of Kida's formula have been proven for fine Selmer groups by D.~Kundu in \cite{kundu2021analogue}. In this particular context, the number fields are assumed to be totally real. Also, such results were proved by J.~Hatley and A.~Lei in the supersingular setting, see \cite{hatley2019arithmetic}.
\subsection*{Acknowledgement} The author is grateful to Debanjana Kundu for pointing out an inaccuracy in the previous draft. He would also like to thank the anonymous referee for a timely and thorough reading of the manuscript and for suggestion various pertinent corrections.

\section{Growth of Iwasawa invariants in towers}
\label{section: growth}
\par In this section, we introduce some preliminary notions and prove the main result of this paper. 

\subsection{Uniform pro-$p$ extensions}
Throughout, $p$ will be a prime $\geq 5$ and $F$ a number field. Let $F_\infty$ be an infinite Galois extension of $F$ with pro-$p$ Galois group $G:=\op{Gal}(F_\infty/F)$. The lower central $p$-series of $G$ is recursively defined as follows: \[G_0:=G\text{ and }G_{n+1}:=G_n^p[G_n, G].\]
\begin{Definition}
The group $G$ is said to be \emph{uniform} if 
\begin{enumerate}
    \item it is finitely generated, 
    \item it is \emph{powerful}, i.e., $[G,G]\subseteq G^p$, 
    \item $[G_n:G_{n+1}]=[G:G_1]$ for all $n\in \Z_{\geq 1}$.
\end{enumerate}
\end{Definition}
Setting $d:=[G:G_1]$, we observe that $[G:G_n]=p^{dn}$. Assume that $G$ has the structure of a $p$-adic Lie group. We say that $F_\infty$ is a \emph{strongly admissible} $p$-adic Lie extension if
\begin{enumerate}
    \item only finitely many primes ramify in $\Fin$,
    \item $\Fin$ contains the cyclotomic $\Z_p$-extension $\Fcyc$ of $F$,
    \item the $p$-torsion subgroup of $G$ is trivial.
    \end{enumerate}
We assume that $F_\infty$ is pro-$p$ a strongly admissible $p$-adic Lie extension and $G$ is uniform.
Note that $G_{n}/G_{n+1}\simeq \left(\Z/p\Z\right)^d$ for $n\in \Z_{\geq 1}$. It is well known that the dimension of $G$ is equal to $d$ and that $G_n=G^{p^n}$, see \cite[Theorem 3.6]{DS}. The extension $\Fin$ is filtered by a tower of number fields. Setting $F^{(n)}:=\Fin^{G_n}$, consider the nonabelian tower
\[F=F^{(0)}\subset F^{(1)}\subset \dots \subset F^{(n)}\subset \dots, \]
and let $F^{(n)}_{\op{cyc}}$ be the cyclotomic $\Z_p$-extension of $F^{(n)}$. We have thus filtered the extension $F_\infty$ into a tower of cyclotomic $\Z_p$-extensions
\[F_{\op{cyc}}=F^{(0)}_{\op{cyc}}\subset F^{(1)}_{\op{cyc}}\subset \dots \subset F^{(n)}_{\op{cyc}}\subset \dots.\]Set $H:=\op{Gal}(\Fin/F_{\op{cyc}})$ and $\Gamma:=G/H\simeq \Z_p$. For $n\in \Z_{\geq 1}$, we write $H_n$ (resp. $\Gamma_n$) for the descending central series of $H$ (resp. $\Gamma$). We list a few useful facts.
\begin{Lemma}\label{H is uni}
The following assertions hold:
\begin{enumerate}
    \item The normal subgroup $H$ is uniform with $(d-1)$ generators, and $H_n$ is identified with $H\cap G_n$,
    \item $\Gamma_n$ is identified with $G_n/H_n$.
\end{enumerate}
\end{Lemma}
\begin{proof}
 See \cite[Theorem 3.6]{DS} and \cite[Lemma 2.6]{HL} for further details.
\end{proof}
As a result, we have that $F^{(n)}_{\op{cyc}}=\Fin^{H_n}$ and hence, \[\op{Gal}(F^{(n+1)}_{\op{cyc}}/F^{(n)}_{\op{cyc}})=H_n/H_{n+1}\simeq \left(\Z/p\Z\right)^{d-1}.\] Since $\Gamma_n=G_n/H_n$, we have that $\Gamma_n=\op{Gal}(F^{(n)}_{\op{cyc}}/F^{(n)})$. We now introduce the Iwasawa-algebra at the $n$-th level, taken to be
\[\Lambda(\Gamma_n):=\varprojlim_L \Z_p[\op{Gal}(L/F^{(n)})],\] where $L$ ranges over all number fields contained in between $F^{(n)}$ and $F^{(n)}_{\op{cyc}}$. Choose a topological generator $\gamma_n$ of $\Gamma_n$ and fix the isomorphism $\Lambda(\Gamma_n)\simeq \Zx$ sending $\gamma_n-1$ to $x$. 
\par More generally, if $\mathcal{G}$ is any pro-$p$ group, set \[\Lambda(\mathcal{G}):=\varprojlim_U \Z_p[\mathcal{G}/U],\] where $U$ ranges over all finite index normal subgroups of $\mathcal{G}$. Given a number field $F$, we set $\Lambda_F:=\Lambda \left(\op{Gal}(F_{\op{cyc}}/F)\right)$.

\subsection{Iwasawa invariants}
\par Let $M$ be a cofinitely generated cotorsion $\Z_p\llbracket x\rrbracket$-module, i.e., the Pontryagin-dual $M^{\vee}:=\op{Hom}(M, \Q_p/\Z_p)$ is a finitely generated and torsion $\Z_p\llbracket x\rrbracket$-module. Recall that a polynomial $f(x)\in \Zx$ is said to be \textit{distinguished} if it is a monic polynomial whose non-leading coefficients are all divisible by $p$. Note that all height $1$ prime ideals of $\Z_p\llbracket x\rrbracket$ are principal ideals $(a)$, where $a=p$ or $a=f(x)$, where $f(x)$ is an irreducible distinguished polynomial. 
According to the structure theorem for $\Zx$-modules (see \cite[Theorem 13.12]{washington1997}), $M^{\vee}$ is pseudo-isomorphic to a finite direct sum of cyclic $\Zx$-modules, i.e., there is a map
\[
M^{\vee}\longrightarrow \left(\bigoplus_{i=1}^s \Zx/(p^{\mu_i})\right)\oplus \left(\bigoplus_{j=1}^t \Zx/(g_j^{e_j}(x)) \right)
\]
with finite kernel and cokernel.
Here, $\mu_i>0$, $e_j>0$ and $g_j(x)$ is an irreducible distinguished polynomial. Furthermore, the numbers $\mu_1, \dots, \mu_s$ and irreducible distinguished polynomials $g_1(x),\dots, g_t(x)$ are uniquely determined.
The characteristic ideal of $M^\vee$ is (up to a unit) generated by
\[
f_{M}^{(p)}(x) = f_{M}(x) := p^{\sum_{i} \mu_i} \prod_j g_j^{e_j}(x).
\]
The $\mu$-invariant of $M$ is defined as the power of $p$ in $f_{M}(x)$.
More precisely,
\[
\mu_p(M):=\begin{cases}
\sum_{i=1}^s \mu_i & \textrm{ if } s>0\\
0 & \textrm{ if } s=0.
\end{cases}
\]
The $\lambda$-invariant of $M$ is the degree of the characteristic element, i.e.,
\[
\lambda_p(M) :=\begin{cases}
\sum_{i=1}^t e_i\deg g_i & \textrm{ if } s>0\\
0 & \textrm{ if } s=0.
\end{cases}
\]
Since the numbers $\mu_i$ and polynomials $g_j(x)$ are uniquely determined by $M$, the $\mu$ and $\lambda$-invariants determined above are well defined. 
\par Let $F$ be a number field and $E$ an elliptic curve over $F$ with good ordinary reduction at all primes of $F$ above $p$. Denote by $\op{Sel}_{p^\infty}(E/F_{\op{cyc}})$ the $p$-primary Selmer group of $E$ over $F_{\op{cyc}}$ (see \cite{coates2000galois} for further details). Suppose $\op{Sel}_{p^\infty}(E/F_{\op{cyc}})$ that is a cotorsion $\Lambda_F$-module, we set $\mu_p(E/F)$ and $\lambda_p(E/F)$ to denote the $\mu$ and $\lambda$-invariant of $\op{Sel}_{p^\infty}(E/F_{\op{cyc}})$ respectively, when viewed as a module over $\Lambda_F$. We fix a strongly admissible pro-$p$, uniform, $p$-adic Lie extension $F_\infty/F$ and let $\op{Sel}_{p^\infty}(E/F_\infty)$ be the Selmer group of $E$ over $F_\infty$ (see \cite{lim2021some} for the definition). Throughout, we make the following assumption.
\begin{Assumption}\label{ass}
With notation as above, assume that $\op{Sel}_{p^\infty}(E/F_\infty)$ satisfies the $\mathfrak{M}_H(G)$-conjecture. In greater detail, set $X(E/F_\infty)$ to be the Pontryagin dual of $\op{Sel}_{p^\infty}(E/F_\infty)$. We assume that 
\[X_f(E/F_\infty):=\frac{X(E/F_\infty)}{X(E/F_\infty)[p^\infty]}\] is finitely generated as a $\Lambda(H)$-module.
\end{Assumption}

\subsection{An analogue of Kida's formula}
\par Hachimori and Matsuno in \cite{HM} proved an analogue of Kida's formula for Selmer groups of elliptic curves. We recall this result and a refinement due to Lim. Let $F$ be a number field and $E_{/F}$ an elliptic curve. Let $L/F$ be a finite Galois extension such that $\op{Gal}(L/F)$ is a $p$-group. Let $P_1(E, L_{\op{cyc}})$ (resp. $P_2(E, L_{\op{cyc}})$) be the set of primes $\eta\nmid p$ of $L_{\op{cyc}}$ that are ramified in the extension $L_{\op{cyc}}/K_{\op{cyc}}$, at which $E$ has split multiplicative reduction (resp. $E$ has good reduction and $E(L_{\op{cyc},\eta})[p]\neq 0$). Given a prime $\eta$ of $L_{\op{cyc}}$, set $e_{L_{\op{cyc}}/K_{\op{cyc}}}(\eta)$ to denote the ramification index of $\eta$ with respect to the extension $L_{\op{cyc}}/K_{\op{cyc}}$.
\begin{Theorem}[M.~F.~Lim]\label{lim kida}
Let $p\geq 5$ be a prime number, $F$ a number field and $E_{/F}$ an elliptic curve with good ordinary reduction at all primes of $F$ above $p$. Let $L/F$ be a Galois extension for which $\op{Gal}(L/F)$ is a $p$-group. Assume that there is a pro-$p$ strongly admissible $p$-adic Lie extension $F_\infty/F$ such that
\begin{enumerate}
    \item $F_\infty$ contains $L$,
    \item Assumption \ref{ass} is satisfied.
\end{enumerate}
Then, the following assertions hold
\begin{enumerate}
    \item $\op{Sel}_{p^\infty}(E/L_{\op{cyc}})$ is cotorsion over the Iwasawa algebra $\Lambda_L$,
    \item $\mu_p(E/L)=[L:K]\mu_p(E/K)$,
    \item \[\begin{split}\lambda_p(E/L)=&[L_{\op{cyc}}:K_{\op{cyc}}]\lambda_p(E/K)+\sum_{\eta\in P_1(E/L_{\op{cyc}})} \left(e_{L_{\op{cyc}}/K_{\op{cyc}}}(\eta)-1\right)\\&+2\sum_{\eta\in P_2(E/L_{\op{cyc}})} \left(e_{L_{\op{cyc}}/K_{\op{cyc}}}(\eta)-1\right).\end{split}\]
\end{enumerate}
\end{Theorem}
\begin{proof}
The result as stated follows from the results in \cite{lim2021some}, as we now explain. Note that since it is assumed that the Selmer group of $E$ over $F_\infty$ satisfies the $\mathfrak{M}_H(G)$-conjecture, it follows from \cite[Proposition 2.5]{CS12} that $\op{Sel}_{p^\infty}(E/L_{\op{cyc}})$ is cotorsion over $\Lambda_L$. First we reduce to the case when $K_{\op{cyc}}\cap L=K$, i.e., $[L:K]=[L_{\op{cyc}}:K_{\op{cyc}}]$. Letting $K'=K_{\op{cyc}}\cap L$, it is easy to see that if the result holds for the extensions $L/K'$ and $K'/K$, then it holds for $L/K$. \par First, we prove that the result holds for $K'/K$. It is a simple exercise to show that $\mu_p(E/K')=[K':K] \mu_p(E/K)$. Furthermore $\lambda_p(E/K')$ is the $\Z_p$-corank of $\op{Sel}_{p^\infty}(E/K'_{\op{cyc}})$. Since $K'$ is contained in $K_{\op{cyc}}$, we have that $K'_{\op{cyc}}=K_{\op{cyc}}$. Therefore, $\lambda_p(E/K')$ is equal to $\lambda_p(E/K)$. Thus, the result is shown to hold for $K'/K$ and it suffices to prove the result for $L/K'$. Upon replacing $K$ with $K'$, we thus reduce to the case when $[L:K]=[L_{\op{cyc}}:K_{\op{cyc}}]$. In this setting the result follows from \cite[Theorem 4.1 \& section 5]{lim2021some}. Indeed the decomposition conditions in \emph{loc. cit.} are equivalent to the conditions on $P_1$ and $P_2$.
\end{proof}

\begin{Remark}
It should be noted here that in section 5 of \cite{lim2021some} an additional assumption is made, namely that $F$ contains the $p$-th roots of unity. This assumption is in place to guarantee the existence of an admissible $p$-adic Lie extension of $L$, see Lemma 4.2 of \emph{loc. cit}. Since we have assumed that such an admissible $p$-adic Lie extension $F_\infty/L$ exists to begin with, there is no need for this additional assumption.
\end{Remark}
\subsection{Main result}
\par Let $E$ be an elliptic curve over a number field $F$ with good ordinary reduction at all primes above $p$. Throughout, we shall make the following assumption.

We introduce some further notation. Let $Q_1=Q_1(E,F_\infty)$ (resp. $Q_2=Q_2(E,F_\infty)$) be the set of primes $w\nmid p$ of $F_{\op{cyc}}$ that are ramified in $F_\infty$, at which $E$ has split multiplicative reduction (resp. $E$ has good reduction and $E(F_{\op{cyc},w})[p]\neq 0$). We stress here that $Q_1$ and $Q_2$ consist of subsets of primes of $F_{\op{cyc}}$ and not $F_\infty$. Recall that it is stipulated that only finitely many primes ramify in $F_\infty$, and since all primes are finitely decomposed in $F_{\op{cyc}}$, it follows that $Q_1$ and $Q_2$ are finite. For $i=1,2$, we set $q_i:=\#Q_i$.
\begin{Theorem}\label{premain}
Let $n$ be a positive integer. Suppose that the conditions of Assumption \ref{ass} hold, then, $\Seln$ is a cotorsion $\Lambda_{F^{(n)}}$-module, with 
\[\mu_p(E/F^{(n)})=p^{nd}\mu_p(E/F).\] Furthermore, we have that
\[p^{n(d-1)}\lambda_p(E/F)\leq \lambda_p(E/F^{(n)})\leq  p^{n(d-1)}\lambda_p(E/F)+\left(p^{n(d-1)}-p^{n(d-2)}\right)(q_1+2q_2).\]

\end{Theorem}
\begin{proof}
According to Theorem \ref{lim kida}, $\Seln$ is a cotorsion module over $\Lambda_{F^{(n)}}$ and the $\mu$-invariant is given by
\[\mu_p(E/F^{(n)})=[F^{(n)}:F]\mu_p(E/F)=p^{nd}\mu_p(E/F).\]
Furthermore, the $\lambda$-invariant is
\[\lambda_p(E/F^{(n)})=[F^{(n)}_{\op{cyc}}:F_{\op{cyc}}]\lambda_p(E/F)+\sum_{w\in P_1} \left(e(w)-1\right)+2\sum_{w\in P_2} \left(e(w)-1\right),\] where $e(w)$ is the ramification index of $w$ in $F^{(n)}_{\op{cyc}}/F_{\op{cyc}}$, $P_1$ and $P_2$ are the set of primes of $F^{(n)}_{\op{cyc}}$ defined as follows:
\[\begin{split}
    & P_1=\{w\mid w\nmid p, E\text{ has split multiplicative reduction at }w\},\\
    & P_2=\{w\mid w\nmid p, E\text{ has good reduction at }w\text{ and }E(\cF_{n,w})\text{ has a point of order }p\}.
\end{split}\]
Since $p$ is odd, the primes $P_i$ lie above $Q_i$, therefore, \[\sum_{w\in P_i} \left(e(w)-1\right)=\sum_{v\in Q_i} \left(\sum_{w\mid v} \left(e(w)-1\right)\right).\]
Choose a prime $w_0\in P_i$ above $v$. Since $e(w)$ is the same for all primes $w|v$, we have that 
\[\sum_{w\mid v} \left(e(w)-1\right)=\left(1-e(w_0)^{-1}\right)\sum_{w\mid v}e(w)\leq \left(1-e(w_0)^{-1}\right)[F^{(n)}_{\op{cyc}}:F_{\op{cyc}}].\]According to Lemma \ref{H is uni}, $H$ is uniform with $d-1$ generators. Note that $[F^{(n)}_{\op{cyc}}:F_{\op{cyc}}]=[H:H_n]=p^{(d-1)n}$. Since pro-$p$ tame inertia is generated by a single element, it follows that $e(w_0)\leq p^{n}$. Putting it all together, the result follows.
 
\end{proof}
\begin{Theorem}\label{main th}
Let $E$ be an elliptic curve defined over a number field $F$ and $\Fin$ a uniform pro-$p$ extension of $F$ satisfying aforementioned conditions and suppose that the Selmer group over $F_{\op{cyc}}$ is cotorsion as a $\Z_p\llbracket x\rrbracket$-module. Then, we have the following bound
\[\op{rank} E(F^{(n)})\leq p^{n(d-1)}\lambda_p(E/F)+\left(p^{n(d-1)}-p^{n(d-2)}\right)(q_1+2q_2).\]
\end{Theorem}
\begin{proof}
The result immediately follows from Theorem \ref{premain} and the inequality 
\[\op{rank} E(F^{(n)})\leq \lambda_p(E/F^{(n)}),\]see \cite[Theorem 1.9]{greenbergITEC}.
\end{proof}

\begin{Remark}\label{better estimate}
The estimate above is stronger than \cite[Theorem 3.1]{HL}. The error term is $O(p^{n(d-2)})$ and their method used relies on the work of M.~Harris. In greater detail, \cite[Theorem 1.10]{harris2000correction} is the key result used in the estimate in \cite[Lemma 3.3]{HL}. Note however, that the error estimate of Hung-Lim is known to be 0 under certain additional constraints. Namely, if certain cohomology groups vanish and the $p$-torsion group $E(F_\infty)(p)$ is finite, see \cite[Theorem 3.2]{HL} and the remark following it. If $F_\infty$ contains the extension $F(E[p^\infty])$ generated by the $p$-primary torsion of $E$, then the error term of Hung-Lim is non-zero, even under additional assumptions. Also, even when the error term of Hung-Lim is 0, the the estimate above is strictly better when $q_1$ or $q_2$ is non-zero.
\end{Remark}
The improvement in the bound has some non-trivial consequences, which we shall explain in the next section. The following is a Corollary to Theorem \ref{premain} and is entirely unconditional.

\begin{Corollary}
Let $E$ and $F_\infty$ be as in Theorem \ref{main th}. Assume that $q_1=q_2=0$ and 
\[\op{Sel}_{p^\infty}(E/F_{\op{cyc}})=0.\] Then, $\op{Sel}_{p^\infty}(E/F^{(n)}_{\op{cyc}})=0$ for all $n$.
\end{Corollary}
\begin{proof}
Since the $\mu$-invariant $\mu_p(E/F)=0$, it follows that $\mathfrak{M}_H(G)$ is satisfied for $F_\infty$, see \cite[Theorem 2.1]{CS12}. By Theorem \ref{premain}, it follows that the $\mu$ and $\lambda$-invariants of $\op{Sel}_{p^\infty}(E/F^{(n)}_{\op{cyc}})$ are $0$, hence, $\op{Sel}_{p^\infty}(E/F^{(n)}_{\op{cyc}})$ is finite. On the other hand, this Selmer does not contain any finite index submodules (see \cite[Proposition 4.14]{Gre99}), hence, must be 0.
\end{proof}
\section{Special Cases}
\par In this section, we study special cases of Theorem \ref{main th}. Assume throughout that the Assumption \ref{ass} is satisfied. Recall from the proof of Theorem \ref{lim kida} that this in particular implies that $\op{Sel}_{p^\infty}(E/L)$ is cotorsion over $\Lambda_L$ for every number field extension $L/K$ contained in $F_\infty$.
\subsection{$\Z_p^d$-extensions} Throughout this subsection, $F$ will be an abelian number field and $E_{/\Q}$ an elliptic curve with good ordinary reduction at $p$. Let $\Fin$ be the composite of all $\Z_p$-extensions of $F$, note that $G=\op{Gal}(F_\infty/F)\simeq \Z_p^d$, where $d=r_2(F)+1$. For instance, when $F$ is an imaginary quadratic field, then, this gives a $\Z_p^2$-extension of $F$. To emphasize the dependence on the prime $p$, we denote the extension by $F_\infty(p)$. On the other hand, it follows from results of K.~Kato and D.~Rohlrich \cite[Theorem 2.2]{HM} that the Selmer group $\op{Sel}_{p^\infty}(E/F_\infty)$ is cotorsion as a $\Zx$-module. It is well known that any $\Z_p$-extension is unramified away from $p$ (see \cite{washington1997}), hence, the composite of such extensions has the same property. Let us state a few Corollaries to Theorem \ref{main th}, the first of which gives a simple criterion for the rank to be zero throughout the $\Z_p^d$-tower.
\begin{Corollary}\label{cor 3.1}
Let $E$ be as above and assume that $\lambda_p(E/F_{\op{cyc}})=0$. Then, \[\op{rank} E(F^{(n)})=0\] for all $n\in \Z_{\geq 1}$ and \[\mu_p(E/F^{(n)})=p^{nd}\mu_p(E/F).\]
\end{Corollary}
\begin{proof}
Note that since $E$ has good ordinary reduction at the primes above $p$ and \[\op{rank} E(F)\leq \lambda_p(E/F),\] the Mordell--Weil rank of $E$ is 0. Since $\Fin$ is unramified at all primes $w\nmid p$, the quantities $q_1$ and $q_2$ in Theorem \ref{main th} are both equal to 0. 
\end{proof}
\begin{Remark}
When $F$ is an imaginary quadratic field and $E$ is a CM elliptic curve over $F$, the result of Hung-Lim in the above context shall imply that the rank in \emph{bounded} in the tower, however, not identically 0. In the more general case, their result implies that the growth is $O(p^{n(d-2)})$ unless certain homology groups are known to vanish, see the discussion after \cite[Theorem 3.2]{HL}.
\end{Remark}
\par \emph{Example:} Picking an elliptic curve $E_{/\Q}$ at random, there are typically some primes at which $E[p]$ is residually reducible as a Galois module. At these primes, it is possible that the $\mu$-invariant $\mu_p(E/\Q)$ does not vanish. For instance, let's pick the elliptic curve of smallest conductor with cremona label $\href{https://www.lmfdb.org/EllipticCurve/Q/11a2/}{11a2}$. We find that $E$ has good ordinary reduction at $5$ with $\mu_5(E/\Q)=2$ and $\lambda_5(E/\Q)=0$. Suppose that there is an imaginary quadratic field $F/\Q$ in which $\op{rank} E(F)=0$ and $\lambda_5(E/F)=0$ as well. Let $F_\infty$ be the $\Z_p^2$-extension of $F$. Then, indeed, since $\mu_5(E/F)\geq 2$, the above result implies that 
\[\mu_5(E/F^{(n)})\geq 2p^{2n}\] for all $n\geq 1$, however, the rank of $E(F_n)$ remains $0$ throughout. Unfortunately, the author is not aware of any existing computer packages that can compute the $\lambda$-invariant over an imaginary quadratic field.

\begin{Corollary}\label{100 percent}
Let $E_{/\Q}$ be an elliptic curve and $F$ an abelian number field satisfying
\begin{enumerate}
    \item $\op{rank} E(F)=0$,
    \item $E$ does not have complex multiplication.
\end{enumerate}
Then, for $100\%$ of primes $p$ at which $E$ has good ordinary reduction, \[\op{rank} E(F_\infty(p))=0.\]
\end{Corollary}
\begin{proof}
In her thesis \cite[Theorem 5.1.1]{kundu}, D.~Kundu generalized a result of R.~Greenberg to show that the proportion of primes $p$ such that
\begin{enumerate}
    \item $E$ has good ordinary reduction at the primes of $F$ above $p$,
    \item $\op{Sel}_{p^\infty}(E/\Fcyc)=0$,
\end{enumerate} is $100\%$. The result follows from this and Corollary \ref{cor 3.1}.
\end{proof}
The following is a special case of \cite[Conjecture 1]{HL}.

\begin{Conjecture}\cite[Conjecture 1']{HL}
Let $E$ be an elliptic curve over an imaginary quadratic field $F$, $p\geq 5$ a prime and $F_\infty$ be the $\Z_p^2$-extension of $F$. Assume that the following conditions are satisfied:
\begin{enumerate}
    \item $E$ has good ordinary reduction at all primes above $p$,
    \item Assumption \ref{ass} is satisfied.
\end{enumerate}Then, we have that $\op{rank} E(F^{(n)})\leq \op{rank} E(F_{\op{cyc}}) p^{n}$ for all $n\geq 1$.
\end{Conjecture} 
\begin{Corollary}
Consider the setting of the above conjecture. Under Assumption \ref{ass}, the Theorem \ref{main th} specializes to give that \[\op{rank} E(F^{(n)})\leq \lambda_p(E/F) p^{n}.\] Thus, the Conjecture is true when 
\[ \lambda_p(E/F)=\op{rank} E(F_{\op{cyc}}).\]
\end{Corollary}
\begin{Remark}
Note that $\lambda_p(E/F)\geq \op{rank} E(F^{(n)})$ for all $n$ and hence,  $\lambda_p(E/F)\geq \op{rank} E(\Fcyc)$. Indeed, it can be expected that $\lambda_p(E/F)=\op{rank} E(F)$ for $100\%$ of primes above which $E$ has good ordinary reduction. There is much evidence pointing towards this expectation for elliptic curves defined over the rationals, see \cite{stats}. We do expect that similar arguments do carry over to elliptic curves over imaginary quadratic fields. 
\end{Remark}

\subsection{False-Tate curve extensions}
Let $\ell$ be a finite set of prime numbers that are coprime to $p$ and let $F_\infty$ be the False-Tate curve extension of $F=\Q(\mu_p)$, given by \[F_\infty:=\Q(\mu_{p^\infty},\ell^{\frac{1}{p^{\infty}}}).\] In other words, it is the extension obtained by adjoining all $p$-power roots of $1$ and $\ell$. It is easy to see that $F_\infty/F$ is a uniform pro-$p$ extension of $F$ of dimension $d=2$. Thus Theorem \ref{main th} specializes to give us that 
\[\op{rank} E(F^{(n)})\leq \lambda_p(E/F)p^n+(p^n-1)\left(q_1+2q_2\right).\] Let us compute the values of $q_1$ and $q_2$ for a given example. We note that it is difficult to compute $\lambda_p(E/F)$ due to the base change to $F=\Q(\mu_p)$. 
\par \emph{Example:} We pick an elliptic curve and prime at random. Let $E=\href{https://www.lmfdb.org/EllipticCurve/Q/11a2/}{11a2}$ in Cremona label and $p=7$. The elliptic curve is defined over $\Q$ and we consider its base change to $F=\Q(\mu_7)$. It follows from Assumption \ref{ass} and from the proof of Theorem \ref{lim kida} that $\op{Sel}_{7^\infty}(E/\Q(\mu_{7^\infty}))$ is cotorsion over $\Lambda_F$. Note that since $F$ is an abelian extension of $\Q$, the cotorsion property of the Selmer group (over $\Lambda_F$) also follows from results due to Kato, see \cite{kato2004p}. The image of the residual representation at $p=7$ contains $\op{SL}_2(\F_7)$, as stated in the link provided. Hence, after base change to $\Q(\mu_7)$ the image of the residual representation will still contain $\op{SL}_2(\F_7)$. It is thus reasonable to expect that the $\mu$-invariant $\mu_7(E/F)=0$, however, this is difficult to prove and needs to be assumed. Consider the False Tate extension
\[F_\infty:=\Q(\mu_{7^\infty},11^{\frac{1}{7^{\infty}}}).\]Then, $E$ has split multiplicative reduction at $11$, hence $q_2=0$, however, $q_1>0$. Since $11^3\equiv 1\mod{7}$ and $11^3\not\equiv 1\mod{49}$, there are precisely two primes above $11$ in $\Fcyc=\Q(\mu_{7^\infty})$. It follows that $q_1=2$. Putting it all together, we find that
\[\op{rank} E(F^{(n)})\leq \lambda_7(E/F)7^n+2(7^n-1).\]

\par Next, we prove a statistical result. 
\begin{Corollary}\label{2ndapp}
Consider the elliptic curve $E=11a2$ from the example above and set $p=7$. There is a positive density set of primes $\ell$ such that \[\op{rank}E\left(\Q(\mu_{7^{n+1}}, \ell^{\frac{1}{7^n}})\right)\leq \lambda_7(E/F)7^n\]for any integer $n\geq 1$.
\end{Corollary}
\begin{proof}
For each prime number $\ell$ such that $\ell\equiv 1\mod{7}$, let $F_\infty^{(\ell)}$ denote the extension $\Q(\mu_{7^\infty}, \ell^{\frac{1}{7^\infty}})$. Note that $\ell$ splits in $\Q(\mu_7)$ and the only prime other than $7$ that ramifies in $F_\infty^{(\ell)}$ is $\ell$. Since $E$ has bad reduction at only finitely many primes, we deduce that $q_1=0$ for all extensions $F_\infty^{(\ell)}$ except for finitely many choices of $\ell$. Recall that $Q_2$ consists of the primes $w\nmid 7$ of $F_{\op{cyc}}$ that are ramified in $F_\infty$, such that $E$ has good reduction at $w$ and $E(F_{\op{cyc},w})[7]\neq 0$. Since the formal group of $E$ at $w$ is pro-$\ell$, $E(F_{\op{cyc},w})[7]\simeq E(k_w)[7]$, where $k_w$ is the residue field of $F_{\op{cyc},w}$. Since $k_w$ is a $7$-extension of $\F_\ell$, it follows that $E(k_w)[7]\neq 0$ if and only if $E(\F_\ell)[7]\neq 0$. Thus, the prime $w$ lies in $Q_2$ if the following conditions are satisfied:
\begin{enumerate}
    \item $w|\ell$, 
    \item $E$ has good reduction at $\ell$, 
    \item $E(\F_\ell)[7]\neq 0$.
\end{enumerate}
This latter condition is satisfied for $\frac{1}{7}$ of all primes $\ell$, see \cite[section 2]{cojocaru2004questions} for further details. We show that a similar application of the Chebotarev density theorem shows that for a positive proportion of the primes $\ell$, both $q_1$ and $q_2$ are $0$. Note that if a prime $\ell$ splits in $\Q(\mu_7)$ precisely when $\ell\equiv 1\mod{7}$. Let $\bar{\rho}:\op{Gal}(\bar{\Q}/\Q)\rightarrow \op{GL}_2(\F_7)$ be the Galois representation on the $7$-torsion points $E[7]$. Denote by $\Q(\bar{\rho})$ the number field fixed by the kernel of $\bar{\rho}$. Note that $\op{det}\bar{\rho}$ is the mod-$7$ cyclotomic character, which we denote by $\bar{\chi}$. Therefore, the field $\Q(\bar{\rho})$ contains $\Q(\mu_7)$, which is the field fixed by the kernel of $\bar{\chi}$. Let $\ell\neq 7$ be a prime at which $E$ has good reduction, note that $\bar{\rho}$ is unramified at $\ell$. Let $\sigma_\ell\in \op{Gal}\left(\Q(\bar{\rho})/\Q\right)$ denote the Frobenius element at $\ell$, it is well known that the characteristic polynomial of $\bar{\rho}(\sigma_\ell)$ is given by
\[\op{det}\left(\op{Id}\cdot x-\bar{\rho}(\sigma_\ell)\right)=x^2-(\ell+1-\#E(\F_\ell))x+\ell.\]
Thus if $\ell$ is a prime such that \begin{equation}\label{mtx eqns} \op{det}\bar{\rho}(\sigma_\ell)=1\text{ and }\op{tr} \bar{\rho}(\sigma_\ell)\neq 2,\end{equation} then, $\ell$ splits in $\Q(\mu_7)$ and $E(\F_\ell)[7]= 0$. We show that there is a positive density set of primes $\ell$ satisfying the above conditions \eqref{mtx eqns}. According to \href{https://www.lmfdb.org/EllipticCurve/Q/11a2/}{the LMFDB Database} \cite{cremona2021functions}, the image of $\bar{\rho}$ contains $\op{SL}_2(\F_7)$. Thus there is $\sigma\in \op{Gal}(\Q(\bar{\rho})/\Q)$ such that $\bar{\rho}(\sigma)=\mtx{2}{1}{1}{1}$, and thus satisfies \eqref{mtx eqns}. According to the Chebotarev density theorem, there is a positive proportion of primes such that $\sigma_\ell=\sigma$, and thus, $\bar{\rho}(\sigma_\ell)=\mtx{2}{1}{1}{1}$. As a result, there is a positive density set of primes $\ell$ such that $q_1$ and $q_2$ are both zero, and this completes the proof.

\end{proof}
\subsection{The field generated by torsion points}
\par We come to the example in which $F_\infty$ is the field $\Q(E[p^\infty])$, i.e., the field generated by the $p$-primary torsion points of $E$. Assume that $E$ does not have complex multiplication. Then, by Serre's Open image theorem (see \cite[section 4, Theorem 3]{Serre72}), $G$ is a finite index subgroup of $\op{GL}_2(\Z_p)$ and it follows from this that the dimension of $G$ is $4$. In this setting, $F=\Q(E[p])$, and $F^{(n)}=\Q(E[p^{n+1}])$. We find that \[\op{rank} E(F_n)\leq \lambda_p(E/F)p^{3n}+q_1(p^{3n}-p^{2n}),\] where $q_1$ is simply the number of primes $\ell\neq p$ at which $E$ has split multiplicative reduction, and $q_2=0$. On the other hand, according to \cite[Remark after Theorem 3.2]{HL} the result of Hung-Lim \cite[Theorem 3.2]{HL} gives 
\[\op{rank} E(F_n)\leq (\lambda_p(E/F)+q_1)p^{3n}+8, \] when the whole dual Selmer group $\mathscr{X}(E/F_\infty)$ is finitely generated over $\Lambda(H)$.
\bibliographystyle{abbrv}
\bibliography{references}

\end{document}